\documentclass[a4paper,11pt,oneside,reqno]{amsart}
\usepackage[T1]{fontenc} \usepackage{lmodern}
\usepackage{amsmath, amssymb, amsthm, mathtools}
\usepackage{mathrsfs}
\usepackage{hyperref}
\usepackage[a4paper, margin=1in]{geometry}
\usepackage{enumitem}
\setlist[enumerate]{label=(\arabic*), ref=(\arabic*)}
\hypersetup{bookmarksnumbered=true}
\numberwithin{equation}{section}

\newtheorem{theorem}{Theorem}[section]

\newtheorem{proposition}[theorem]{Proposition}
\newtheorem{lemma}[theorem]{Lemma}

\theoremstyle{definition}
\newtheorem{definition}[theorem]{Definition}
\newtheorem{assumption}[theorem]{Assumption}
\theoremstyle{remark}
\newtheorem{remark}[theorem]{Remark}

\newcommand{\R}{\mathbb{R}}
\newcommand{\haus}{\mathcal{H}}
\newcommand{\Hom}{\mathrm{Hom}}
\newcommand{\G}{\mathbf{G}}
\newcommand{\loc}{\mathrm{loc}}
\newcommand{\wt}[1]{\mathord{\lVert #1 \rVert}}
\newcommand{\V}{\mathbf{V}}
\newcommand{\IV}{\mathbf{IV}}
\newcommand{\var}{\mathbf{var}}
\newcommand{\VarTan}{\mathrm{VarTan}}
\newcommand{\C}{\mathbf{C}}
\newcommand{\g}{\mathrm{g}}
\renewcommand{\c}{\mathbf{c}}
\newcommand{\id}{\mathrm{id}}
\DeclarePairedDelimiter{\abs}{\lvert}{\rvert}

\DeclareMathOperator{\clos}{clos}
\DeclareMathOperator{\spt}{spt}
\DeclareMathOperator{\dist}{dist}
\DeclareMathOperator{\Lip}{Lip}
\DeclareMathOperator{\tr}{tr}

\begin{document}
\title[Non-uniqueness of Brakke flows]{Non-uniqueness of Brakke flows starting from minimal surfaces with singularities}
\author{Kotaro Motegi}
\address{Department of Mathematics, Institute of Science Tokyo, 2-12-1 Ookayama, Meguro-ku, Tokyo 152-8551, Japan}
\email{motegi.k.3c77@m.isct.ac.jp}

\begin{abstract}
We prove the existence of a genuinely time-dependent Brakke flow starting from $\Gamma_0 \subset \R^{n+1}$ whose associated multiplicity-one varifold is stationary, provided that, at some singular point, the scale-invariant $L^2$ distance of $\Gamma_0$ from an $n$-dimensional plane has sufficiently small limsup as the scale tends to zero.
This yields the dynamical instability of $\Gamma_0$, a notion recently introduced by Stuvard and Tonegawa, and hence the non-uniqueness of Brakke flows starting from $\Gamma_0$.
A notable feature of our result is that it holds without assuming the uniqueness of tangent cones at the singular point.
\end{abstract}

\maketitle

\section{Introduction}
\label{sec:introduction}

A family of hypersurfaces $\{\Gamma_t\}_{t \geq 0}$ is called a mean curvature flow (hereafter abbreviated as MCF) if its normal velocity $v$ is equal to the mean curvature $h$ of $\Gamma_t$ at each point and time.
MCF arises naturally as a model for the motion of grain boundaries in metals.
Since grain boundaries are not necessarily smooth and may have singularities, classical smooth solutions are insufficient to describe their evolution.
To describe the evolution of such singular surfaces, Brakke introduced in \cite{Bra78} a weak formulation of MCF, now known as Brakke flow.
Within this framework, hypersurfaces are replaced by varifolds, which are a measure-theoretic generalization of submanifolds that allows singularities, and the PDE $v = h$ is replaced by an inequality describing the decrease of area along the flow.
Given a smooth hypersurface $\Gamma_0$, there exists a unique smooth MCF starting from $\Gamma_0$, at least for a short time.
However, if $\Gamma_0$ has singularities, uniqueness may fail in general.
For example, although the union of two transverse lines is a static Brakke flow, there exists another Brakke flow starting from the same initial datum in which the quadruple junction breaks up into two triple junctions.

To investigate this phenomenon, Stuvard and Tonegawa introduced in \cite{ST25} a notion of instability for minimal surfaces, called \emph{dynamical instability}.
Let $\Gamma_0$ be a closed countably $n$-rectifiable set in a strictly convex domain in $\R^{n+1}$ such that the associated multiplicity-one varifold is stationary.
We say that $\Gamma_0$ is dynamically unstable if there exists a non-trivial Brakke flow with the same fixed boundary and initial datum.
Here, non-triviality means that the total mass of the flow is strictly smaller than the initial mass $\haus^n(\Gamma_0)$ for every $t > 0$, which ensures that the surface genuinely moves immediately after the initial time, rather than remaining stationary for some positive time interval.
Since $\Gamma_0$ is stationary, the static flow associated with $\Gamma_0$ is itself a Brakke flow.
Hence, the dynamical instability of $\Gamma_0$ implies the non-uniqueness of Brakke flows starting from $\Gamma_0$.
Since smooth minimal surfaces are expected to be dynamically stable as a consequence of classical regularity and uniqueness theory for smooth MCF, Stuvard and Tonegawa asked which types of singularities force dynamical instability.
Their main result shows that $\Gamma_0$ is dynamically unstable if, at some singular point $x_0 \in \Gamma_0$, $\Gamma_0$ has a tangent cone given by a plane with multiplicity at least two, and if the rescalings $(\Gamma_0 - x_0)/r$ converge to this plane at a rate faster than $(\log(1/r))^{-1/2}$ as $r \to 0$.

In the present paper, we prove dynamical instability for a broader class of stationary varifolds.
The following is our main result.

\begin{theorem}
    \label{thm:intro}
    For any $\Theta_0 > 1$, there exists a constant $\mu_0 = \mu_0(n,\Theta_0) \in (0,1)$ such that the following holds.
    Suppose that a closed countably $n$-rectifiable set $\Gamma_0$ satisfies:
    \begin{enumerate}
        \item The multiplicity-one varifold associated to $\Gamma_0$ is stationary;
        \item The density $\Theta^n(\haus^n\lfloor_{\Gamma_0},0)$ of $\haus^n\lfloor_{\Gamma_0}$ at the origin satisfies $1 < \Theta^n(\haus^n\lfloor_{\Gamma_0},0) \leq \Theta_0$;
        \item There exists an $n$-dimensional plane $T$ such that
            \begin{equation}
                \label{eq:intro_assump}
                \limsup_{R \to 0} R^{-n-2}\int_{\Gamma_0 \cap U_R} \abs{T^\perp x}^2 \,d\haus^n(x) \leq \mu_0^2.
            \end{equation}
    \end{enumerate}
    Then there exists a non-trivial Brakke flow starting from $\Gamma_0$.
\end{theorem}

The precise statement is given in Theorem~\ref{thm:main}.
A significant feature of this result is that it does not require the uniqueness of tangent cones.
Indeed, \eqref{eq:intro_assump} is satisfied whenever all tangent cones have sufficiently small $L^2$ distance from $T$ on the unit ball (see Remark~\ref{rem:assumption}).
As an important special case, if $\Gamma_0$ has a unique tangent cone $\C$ at the origin, and $\C$ is either a plane with multiplicity at least two or a union of half-planes meeting along a common axis at sufficiently small angles, then \eqref{eq:intro_assump} holds for some plane, and hence $\Gamma_0$ is dynamically unstable.
Therefore, Theorem~\ref{thm:intro} removes the logarithmic decay assumption from \cite{ST25}, and broadens the class of singularities for which dynamical instability follows.

There are several situations in which the assumptions of Theorem~\ref{thm:intro} are automatically satisfied.
For stationary two-valued Lipschitz graphs, Becker-Kahn proved in \cite{BK17} the uniqueness of tangent cones that are either a pair of transverse planes or a union of four half-planes meeting along a common axis.
More recently, Becker-Kahn, Minter, and Wickramasekera \cite{BKMW25} established an $\varepsilon$-regularity theorem for stationary integral varifolds that are close to a multiplicity-two plane and satisfy a certain topological structural condition; in particular, their result implies the uniqueness of flat tangent cones for stationary two-valued Lipschitz graphs.
In the setting of stable codimension-one integral varifolds, Minter and Wickramasekera proved in \cite{MW24} that, in the absence of classical singularities of density less than $Q$, tangent cones that are either a multiplicity-$Q$ plane or a union of $2Q$ half-planes meeting along a common axis are unique.
Subsequently, Minter \cite{Min24} proved the uniqueness of tangent cones consisting of $2Q+1$ half-planes meeting along a common axis for the same class of stable codimension-one integral varifolds.
Since every classical singularity of a codimension-one area minimizing current mod $p$ has density at least $p/2$, the results of \cite{MW24} apply to this setting with $Q = p/2$.
In particular, a multiplicity-$p/2$ plane or a union of $p$ half-planes meeting along a common axis is a unique tangent cone, see also \cite{DLHMS26}.
Thus, by the results above, $\Gamma_0$ is dynamically unstable whenever it belongs to one of these classes and admits, at a singular point, a tangent cone that is either a plane with multiplicity or a union of half-planes meeting along a common axis at sufficiently small angles.

Next, we explain the main idea of the proof.
Our general strategy follows that of \cite{ST25}.
For each $\varepsilon > 0$, we construct a suitable approximation of $\Gamma_0$ and a corresponding Brakke flow $\{V^\varepsilon_t\}_{t\geq 0}$ starting from it.
The main difficulty is to ensure that the limit flow is non-trivial.
For this, it suffices to establish a uniform upper bound for the density ratio of $V^\varepsilon_t$ in the ball $U_{\sqrt{t+\varepsilon^2}}$ on a short time interval independent of $\varepsilon$.
Brakke's expanding holes lemma provides an estimate of the density ratio at later times in terms of its value at earlier times and an error term involving the scale-invariant $L^2$ distance from a plane.
In \cite{ST25}, the logarithmic decay assumption is used to control this error term.
Under our weaker assumption \eqref{eq:intro_assump}, however, this argument does not directly apply, since the error term may fail to be integrable near the initial time.
Our key idea is to exploit the integrality of the varifolds.
If the $L^2$ distance from a plane is sufficiently small, then the density ratio must be close to an integer.
We use this observation to rule out the possibility that the desired density ratio bound breaks down arbitrarily close to the initial time.
Indeed, suppose that the bound fails at arbitrarily small times.
For each $\varepsilon > 0$, let $T_\varepsilon$ denote the supremum of the times up to which the density ratio remains below $1 + \delta$, where $\delta \in (0,1)$ is a small constant.
If $\limsup_{\varepsilon \to 0} T_\varepsilon = 0$, we rescale the flows so that the critical time $T_\varepsilon$ is normalized to $1$, and pass to a limit Brakke flow.
Then the density ratio attains the threshold $1 + \delta$ at the time $1$, while the $L^2$ closeness inherited from \eqref{eq:intro_assump} forces the density ratio to remain close to $1$, yielding a contradiction.

The paper is organized as follows.
Section \ref{sec:preliminaries} introduces the notation and terminology used throughout the paper.
In Section \ref{sec:main_result}, we state the assumptions on the initial hypersurface and present the precise statement of the main result.
Section \ref{sec:hole_nucleation} is devoted to the construction of a suitable approximation of the initial surface.
In Section \ref{sec:expanding_holes}, we establish Brakke's expanding holes lemma.
Section \ref{sec:monotonicity} applies Huisken's monotonicity formula to derive uniform-in-time estimates for the density ratio and the scale-invariant $L^2$ distance.
Finally, Section \ref{sec:proof} contains the proof of the main result.

\subsection*{Acknowledgment}
The author would like to express his gratitude to his supervisor, Professor Yoshihiro Tonegawa, for suggesting this problem and his continuous encouragement.
This research was supported by the Science Tokyo Support Program for Doctoral Students, funded by the Universities for International Research Excellence.

\section{Preliminaries}
\label{sec:preliminaries}

\subsection{Basic notation}
\label{subsec:notation}

Throughout this paper, $n$ denotes a positive integer.
For $x \in \R^{n+1}$ and $R > 0$, let $U_R(x)$ denote the open ball in $\R^{n+1}$ centered at $x$ with radius $R$.
We write $U_R$ for $U_R(0)$.
The symbol $\haus^n$ denotes the $n$-dimensional Hausdorff measure in $\R^{n+1}$, and $\omega_n$ denotes the volume of the unit ball in $\R^n$.
For a subset $A \subset \R^{n+1}$, let $\clos A$ denote the closure of $A$ in $\R^{n+1}$.
Let $\G(n+1,n)$ be the set of all $n$-dimensional subspaces of $\R^{n+1}$.
We regard $\G(n+1,n)$ as a subset of $\Hom(\R^{n+1},\R^{n+1})$ by identifying each $S \in \G(n+1,n)$ with the orthogonal projection onto $S$.
For $A, B \in \Hom(\R^{n+1},\R^{n+1})$, we define their inner product by $A \cdot B = \tr(A^\ast \circ B)$, where $A^\ast$ denotes the adjoint of $A$.
We write $\abs{\cdot}$ for the norm induced by this inner product.
For $x \in \R^{n+1}$ and $R > 0$, let $\eta_{x,R} \colon \R^{n+1} \to \R^{n+1}$ be the map defined by $\eta_{x,R}(y) = (y-x)/R$.

\subsection{Varifolds}
\label{subsec:varifolds}

Let $U \subset \R^{n+1}$ be an open set, and define $G_n(U) = U \times \G(n+1,n)$.
A $n$-varifold in $U$ is a Radon measure on $G_n(U)$, and set of all $n$-varifolds in $U$ is denoted by $\V_n(U)$.
For $V \in \V_n(U)$, its weight measure $\wt{V}$ is the Radon measure on $U$ defined by
\begin{equation*}
    \wt{V}(\phi) = \int_{G_n(U)} \phi(x) \,dV(x,S)
\end{equation*}
for every $\phi \in C_c(U)$.
We say that $V \in \V_n(U)$ is integral if there exist an $\haus^n$-measurable, countably $n$-rectifiable set $\Gamma \subset U$ and a non-negative, integer-valued, locally $\haus^n$-integrable function $\theta$ on $\Gamma$ such that
\begin{equation}
    \label{eq:varifolds_integral}
    V(\phi) = \int_\Gamma \phi(x,T_x\Gamma)\theta(x) \,d\haus^n(x)
\end{equation}
for every $\phi \in C_c(G_n(U))$, where $T_x\Gamma$ denotes the approximate tangent space of $\Gamma$ at $x$, which exists for $\haus^n$-a.e. $x \in \Gamma$.
The varifold defined by \eqref{eq:varifolds_integral} is denoted by $\var(\Gamma,\theta)$.
We write the set of all integral $n$-varifolds in $U$ as $\IV_n(U)$.

The first variation of $V \in \V_n(U)$ is the linear functional $\delta V$ on $C^1_c(U;\R^{n+1})$ defined by
\begin{equation*}
    \delta V(g) = \int_{G_n(U)} \nabla g(x) \cdot S \,dV(x,S)
\end{equation*}
for every $g \in C^1_c(U;\R^{n+1})$.
When $\delta V$ extends to a continuous linear functional on $C_c(U;\R^{n+1})$, we say that $V$ has locally bounded first variation, and denote the corresponding total variation measure by $\wt{\delta V}$.
If $\wt{\delta V}$ is absolutely continuous with respect to $\wt{V}$, there exists a unique vector field $h(V,\cdot) \in L^1_\loc(\wt{V};\R^{n+1})$ such that
\begin{equation}
    \label{eq:mean_curvature}
    \delta V(g) = -\int_U g(x) \cdot h(V,x) \,d\wt{V}(x)
\end{equation}
for all $g \in C^1_c(U;\R^{n+1})$.
We call $h(V,\cdot)$ the generalized mean curvature of $V$.
For any $V \in \IV_n(U)$, Brakke's perpendicularity theorem \cite[5.8]{Bra78} says that the generalized mean curvature is perpendicular to the tangent space almost everywhere, that is,
\begin{equation}
    \label{eq:perpendicularity}
    S(h(V,x)) = 0 \qquad \text{for $V$-a.e. $(x,S) \in G_n(U)$.}
\end{equation}

A varifold is said to be stationary if its generalized mean curvature vanishes.
If $V \in \IV_n(U)$ is stationary, then the monotonicity formula \cite[17.5]{Sim83} implies that the density
\begin{equation*}
    \Theta^n(\wt{V},x) = \lim_{r \to 0} \frac{\wt{V}(U_r(x))}{\omega_nr^n}
\end{equation*}
exists at every $x \in U$.
Moreover, for every sequence $r_j \to 0$, there exist a subsequence $r_{j'}$ and a stationary $n$-varifold $\C \in \IV_n(\R^{n+1})$ such that $(\eta_{x,r_{j'}})_\sharp V$ converges to $\C$ as Radon measures on $G_n(\R^{n+1})$ as $j' \to \infty$.
The limit $\C$ is a cone, namely, $(\eta_{0,\lambda})_\sharp\C = \C$ for all $\lambda > 0$.
Any such cone $\C$ is called a tangent cone of $V$ at $x$.
The set of all tangent cones of $V$ at $x$ is denoted by $\VarTan(V,x)$.

\subsection{Brakke flows}
\label{subsec:Brakke_flows}

In this subsection, we recall the definition of Brakke flow.
Let $U \subset \R^{n+1}$ be an open set, and let $I \subset \R$ be an interval.

\begin{definition}
    \label{def:Brakke_flow}
    We say that a family of $n$-varifolds $\{V_t\}_{t \in I}$ in $U$ is a $n$-dimensional Brakke flow in $U$ if the following hold:
    \begin{enumerate}
        \item \label{itm:Brakke_integrality}
            For a.e. $t \in I$, $V_t \in \IV_n(U)$;
        \item \label{itm:Brakke_first_variation}
            For a.e. $t \in I$, $V_t$ has locally bounded first variation and $\wt{\delta V_t} \ll \wt{V_t}$;
        \item \label{itm:Brakke_mean_curvature}
            For any $\tilde{U} \subset\subset U$ and $\tilde{I} \subset\subset I$,
            \begin{equation*}
                \sup_{t \in \tilde{I}} \wt{V_t}(\tilde{U}) + \int_{\tilde{I}}\int_{\tilde{U}} \abs{h(V_t,x)}^2 \,d\wt{V_t}(x)dt < \infty;
            \end{equation*}
        \item \label{itm:Brakke_ineq}
            For all $t_1,t_2 \in I$ with $t_1 < t_2$ and $\phi \in C^1_c(U \times I;[0,\infty))$,
            \begin{multline}
                \label{eq:Brakke_ineq}
                \int_U \phi(x,t_2) \,d\wt{V_{t_2}}(x) - \int_U \phi(x,t_1) \,d\wt{V_{t_1}}(x) \\
                \leq \int_{t_1}^{t_2}\int_U (-\phi(x,t)h(V_t,x) + \nabla\phi(x,t)) \cdot h(V_t,x) + \partial_t\phi(x,t) \,d\wt{V_t}(x)dt.
            \end{multline}
    \end{enumerate}
    Furthermore, if $\partial U$ is not empty and $\Sigma \subset \partial U$, we say that $\{V_t\}_{t \in I}$ has fixed boundary $\Sigma$ if, together with conditions \ref{itm:Brakke_integrality}--\ref{itm:Brakke_ineq} above, it holds
    \begin{enumerate}[resume]
        \item For all $t \in I$, $(\clos(\spt\wt{V_t})) \setminus U = \Sigma$.
    \end{enumerate}
\end{definition}

\section{Main result}
\label{sec:main_result}

We begin by stating the assumptions needed for the main theorem.

\begin{assumption}
    \label{asm:initial_surface}
    Let us fix an integer $N \geq 2$.
    Suppose that $U$, $\Gamma_0$, and $\{E_{0,i}\}_{i=1}^N$ satisfy the following conditions:
    \begin{enumerate}
        \item \label{itm:initial_U}
            $U \subset \R^{n+1}$ is a strictly convex bounded domain with boundary $\partial U$ of class $C^2$;
        \item \label{itm:initial_Gamma}
            $\Gamma_0 \subset U$ is a relatively closed, countably $n$-rectifiable set with $\haus^n(\Gamma_0) < \infty$;
        \item \label{itm:initial_E}
            $E_{0,1},\dots,E_{0,N}$ are non-empty, open, and mutually disjoint subsets of $U$ such that $U \setminus \Gamma_0 = \bigcup_{i=1}^N E_{0,i}$;
        \item \label{itm:initial_partial}
            $\partial\Gamma_0 \coloneqq \clos(\Gamma_0) \setminus U$ is not empty, and for each $x \in \partial\Gamma_0$ there exist at least two indexes $i_1 \neq i_2$ in $\{1,\dots,N\}$ such that $x \in \clos(\clos(E_{0,i_j}) \setminus (U \cup \partial\Gamma_0))$ for $j=1,2$;
        \item \label{itm:initial_continuity}
            $\haus^n(\Gamma_0 \setminus \bigcup_{i=1}^N \partial^\ast E_{0,i}) = 0$, where $\partial^\ast E_{0,i}$ is the reduced boundary of $E_{0,i}$.
    \end{enumerate}
\end{assumption}

Conditions~\ref{itm:initial_U}--\ref{itm:initial_partial} are identical to those in \cite[Assumption~1.1]{ST21}.
By \cite[Theorem~2.2]{ST21}, these conditions ensure the existence of a Brakke flow $\{V_t\}_{t \geq 0}$ with fixed boundary $\partial\Gamma_0$, and $\wt{V_0} = \haus^n\lfloor_{\Gamma_0}$.
Condition~\ref{itm:initial_continuity} further guarantees the continuity of the flow at the initial time, namely, $\wt{V_0} = \lim_{t \to 0}\wt{V_t}$.

The next assumption plays a central role in proving dynamical instability.

\begin{assumption}
    \label{asm:tangent_cone}
    Let $\Theta_0 > 1$, $\mu \in (0,1)$, and let $U$, $\Gamma_0$, and $\{E_{0,i}\}_{i=1}^N$ satisfy Assumption~\ref{asm:initial_surface} and $0 \in \Gamma_0$.
    Set $V_0 = \var(\Gamma_0,1)$.
    We suppose that
    \begin{enumerate}
        \item \label{itm:tangent_cone_stationarity}
            $V_0$ is a stationary varifold with $\spt\wt{V_0} = \Gamma_0$;
        \item \label{itm:tangent_cone_density}
            $1 < \Theta^n(\wt{V_0},0) \leq \Theta_0$;
        \item \label{itm:tangent_cone_distance}
            There exists an $n$-dimensional plane $T \in \G(n+1,n)$ such that
            \begin{equation}
                \label{eq:tangent_cone}
                \limsup_{R \to 0} R^{-n-2}\int_{U_R} \abs{T^\perp x}^2 \,d\wt{V_0}(x) \leq \mu^2.
            \end{equation}
    \end{enumerate}
\end{assumption}

\begin{remark}
    \label{rem:assumption}
    For a stationary $V_0 \in \IV_n(U)$, \eqref{eq:tangent_cone} holds if and only if $\VarTan(V_0,0)$ is contained in the set of all stationary cones $\mathbf{C} \in \IV_n(\R^{n+1})$ satisfying
    \begin{equation*}
        \int_{U_1} \abs{T^\perp x}^2 \,d\wt{\mathbf{C}}(x) \leq \mu^2.
    \end{equation*}
    In particular, Assumption~\ref{asm:tangent_cone} does not require the uniqueness of tangent cones.
\end{remark}

The following is the main theorem of this paper.

\begin{theorem}
    \label{thm:main}
    For any $\Theta_0 > 1$, there exists $\mu_0 = \mu_0(n,\Theta_0) \in (0,1)$ such that if Assumption~\ref{asm:tangent_cone} holds for $\mu = \mu_0$, then there exists a Brakke flow $\{V_t\}_{t \geq 0}$ with fixed boundary $\partial\Gamma_0$ such that
    \begin{enumerate}
        \item $\lim_{t \searrow 0} \wt{V_t} = \wt{V_0} = \haus^n \lfloor_{\Gamma_0}$;
        \item $\wt{V_t}(U) < \wt{V_0}(U)$ for all $t > 0$.
    \end{enumerate}
\end{theorem}

\section{Hole nucleation}
\label{sec:hole_nucleation}

As a first step in the proof of the main theorem, we construct a modified initial surface $\Gamma^\varepsilon_0$ by, roughly speaking, creating a hole at the origin.
The construction of $\Gamma^\varepsilon_0$ is the same as that in \cite[Lemma~4.1]{ST25}.
However, due to the different assumption on $\Gamma_0$, we show analogous properties of $\Gamma^\varepsilon_0$ in a different form.

\begin{lemma}
    \label{lem:hole_nucleation}
    There exists $\mu_1 = \mu_1(n) \in (0,1)$ with the following property.
    Suppose that $\Theta_0$, $\mu$, $U$, $\Gamma_0$, $\{E_{0,i}\}_{i=1}^N$, and $T$ are as in Assumption~\ref{asm:tangent_cone} and $\mu \in (0,\mu_1)$.
    Then there exist $\varepsilon_0 = \varepsilon_0(\mu,\Gamma_0,T) \in (0,\dist(0,\partial U)/2)$ and $c = c(n) \in (0,\infty)$ such that, for all $\varepsilon \in (0,\varepsilon_0]$, there exist a relatively closed and countably $n$-rectifiable set $\Gamma^\varepsilon_0 \subset U$, a family $\{E^\varepsilon_{0,i}\}_{i=1}^N$ of pairwise disjoint non-empty open subsets of $U$ with finite perimeter such that:
    \begin{enumerate}
        \item \label{itm:hole_nucleation_1}
            $\Gamma^\varepsilon_0 \setminus U_{2\varepsilon} = \Gamma_0 \setminus U_{2\varepsilon}$ and $E^\varepsilon_{0,i} \setminus U_{2\varepsilon} = E_{0,i} \setminus U_{2\varepsilon}$ for each $i = 1,\dots,N$;
        \item \label{itm:hole_nucleation_2}
            $\Gamma^\varepsilon_0 = U \setminus \bigcup_{i=1}^N E^\varepsilon_{0,i}$;
        \item \label{itm:hole_nucleation_excess}
            For all $R \in (0,\varepsilon_0]$, we have
            \begin{equation}
                \label{eq:hole_nucleation_excess}
                R^{-n-2}\int_{\Gamma^\varepsilon_0 \cap U_R} \abs{T^\perp x}^2 \,d\haus^n(x) \leq c(n)\mu^2;
            \end{equation}
        \item \label{itm:hole_nucleation_density}
            For all $R \in (0,\varepsilon_0]$, we have
            \begin{equation*}
                \label{eq:hole_nucleation_density}
                \frac{\haus^n(\Gamma^\varepsilon_0 \cap U_R)}{\omega_nR^n} \leq c(n)\Theta_0;
            \end{equation*}
        \item \label{itm:hole_nucleation_5}
            $\{\abs{Tx} < \varepsilon\} \cap U_{2\varepsilon} \cap \Gamma^\varepsilon_0 \subset T$.
    \end{enumerate}
\end{lemma}

\begin{proof}
    Without loss of generality, we may assume that $T = \R^n \times \{0\}$, and we write $x = (x',x_{n+1}) \in \R^{n+1} = T \oplus T^\perp$.
    By \cite[7.5 (2) and (6)]{All72} and Assumption~\ref{asm:tangent_cone}, there exist $\mu_1 = \mu_1(n) \in (0,1)$ and $\varepsilon_0 = \varepsilon_0(\mu,\Gamma_0,T) \in (0,1)$ such that for all $\varepsilon, R \in (0,2\varepsilon_0]$, we have
    \begin{gather}
        \label{eq:hole_nucleation_1}
        \Gamma_0 \cap U_{2\varepsilon} \subset \{\abs{x_{n+1}} \leq \varepsilon/20\}, \\
        \label{eq:hole_nucleation_excess_0}
        R^{-n-2}\int_{\Gamma_0 \cap U_R} \abs{x_{n+1}}^2 \,d\haus^n(x) \leq 2\mu^2, \\
        \label{eq:hole_nucleation_density_0}
        \haus^n(\Gamma_0 \cap U_R) \leq 2\Theta_0\omega_nR^n.
    \end{gather}
    Let $\g \colon \R^{n+1} \to \R^{n+1}$, $\Gamma^\varepsilon_0$, and $\{E^\varepsilon_{0,i}\}_{i=1}^N$ be as in \cite[Lemma~4.1]{ST25}.
    The proofs of \ref{itm:hole_nucleation_1}, \ref{itm:hole_nucleation_2}, and \ref{itm:hole_nucleation_5} are identical to those given in \cite[Lemma~4.1]{ST25}, so we only prove \ref{itm:hole_nucleation_excess} and \ref{itm:hole_nucleation_density}.
    Define $\g_\varepsilon(x) = \varepsilon\g(x/\varepsilon)$ for $x \in \R^{n+1}$, and let
    \begin{align*}
        A &= \{(x',x_{n+1}) : \abs{x_{n+1}} \geq 1/5\} \cup \{(x',x_{n+1}) : \abs{x_{n+1}} \leq \abs{x'} - 1\}, \\
        B &= \{(x',0) : \abs{x'} \leq 1\} \cup \{(x',x_{n+1}) : \abs{x'}-1 = \abs{x_{n+1}} \leq 1/5\}.
    \end{align*}
    Set $A_\varepsilon = \eta_{0,1/\varepsilon}(A)$ and $B_\varepsilon = \eta_{0,1/\varepsilon}(B)$.
    By \eqref{eq:hole_nucleation_1} and \cite[(4.2)--(4.7)]{ST25}, we have
    \begin{equation}
        \label{eq:hole_nucleation_2}
        \Lip\g_\varepsilon \leq 2, \qquad \g_\varepsilon|_{A_\varepsilon} = \id_{A_\varepsilon}, \qquad \g_\varepsilon(\Gamma_0 \setminus A_\varepsilon) \subset B_\varepsilon, \qquad \Gamma^\varepsilon_0 \subset \g_\varepsilon(\Gamma_0).
    \end{equation}
    We first prove \ref{itm:hole_nucleation_density}.
    Since there exists $c = c(n) > 0$ such that $\haus^n(B_\varepsilon \cap U_R) \leq c(n)\omega_nR^n$ for all $R > 0$, \eqref{eq:hole_nucleation_density_0} and \eqref{eq:hole_nucleation_2} yield that
    \begin{align*}
        \frac{\haus^n(\Gamma^\varepsilon_0 \cap U_R)}{\omega_nR^n} &\leq \frac{\haus^n(\g_\varepsilon(\Gamma_0) \cap U_R)}{\omega_nR^n} \leq \frac{\haus^n(\g_\varepsilon(\Gamma_0 \cap A_\varepsilon) \cap U_R)}{\omega_nR^n} + \frac{\haus^n(\g_\varepsilon(\Gamma_0 \setminus A_\varepsilon) \cap U_R)}{\omega_nR^n} \\
        &\leq \frac{\haus^n(\Gamma_0 \cap A_\varepsilon \cap U_R)}{\omega_nR^n} + \frac{\haus^n(B_\varepsilon \cap U_R)}{\omega_nR^n} \leq 2\Theta_0 + c(n) \leq c(n)\Theta_0.
    \end{align*}
    for all $R \in (0,\varepsilon_0]$.
    We next prove \ref{itm:hole_nucleation_excess}.
    Fix $R \in (0,\varepsilon_0]$.
    If $R < \varepsilon$, then the left-hand side of \eqref{eq:hole_nucleation_excess} is $0$.
    Hence it suffices to consider the case $R \geq \varepsilon$.
    Since $\R^{n+1} \setminus A_\varepsilon \subset U_{2\varepsilon} \subset U_{2R}$, it follows from \eqref{eq:hole_nucleation_excess_0} and \eqref{eq:hole_nucleation_2} that 
    \begin{align*}
        \int_{\Gamma^\varepsilon_0 \cap U_R} \abs{x_{n+1}}^2 \,d\haus^n(x) &\leq \int_{\g_\varepsilon(\Gamma^\varepsilon_0 \cap A_\varepsilon) \cap U_R} \abs{x_{n+1}}^2 \,d\haus^n(x) + \int_{g_\varepsilon(\Gamma_0 \setminus A_\varepsilon) \cap U_R} \abs{x_{n+1}}^2 \,d\haus^n(x) \\
        &\leq \int_{\Gamma^\varepsilon_0 \cap A_\varepsilon \cap U_R} \abs{x_{n+1}}^2 \,d\haus^n(x) + 2^{n+2}\int_{\Gamma_0 \setminus A_\varepsilon} \abs{x_{n+1}}^2 \,d\haus^n(x) \\
        &\leq (1+2^{n+2})\int_{\Gamma_0 \cap U_{2R}} \abs{x_{n+1}}^2 \,d\haus^n(x) \leq c(n)\mu^2R^{n+2}.
    \end{align*}
    This completes the proof.
\end{proof}

Since $\Gamma^\varepsilon_0$ also satisfies Assumption~\ref{asm:initial_surface}, we can apply the existence theorem for Brakke flows with fixed boundary to obtain a Brakke flow starting from $\Gamma^\varepsilon_0$.

\begin{proposition}[{\cite[Theorems~2.2 and 2.3]{ST21}}]
    \label{prop:existence}
    With $\Gamma^\varepsilon_0$ and $\{E^\varepsilon_{0,i}\}_{i=1}^N$ given in Lemma~\ref{lem:hole_nucleation}, there exists a Brakke flow $\{V^\varepsilon_t\}_{t \geq 0}$ with fixed boundary $\partial\Gamma_0$ and $\wt{V^\varepsilon_0} = \haus^n\lfloor_{\Gamma^\varepsilon_0}$ such that
    \begin{equation*}
        \wt{V^\varepsilon_t}(U) + \int_0^t\int_U \abs{h(V^\varepsilon_t,\cdot)}^2 \,d\wt{V^\varepsilon_t}dt \leq \wt{V^\varepsilon_0}(U)
    \end{equation*}
    for all $t \geq 0$.
    For each $i = 1,\dots,N$, there exists a one-parameter family $\{E^\varepsilon_i(t)\}_{t \geq 0}$ of open sets $E^\varepsilon_i(t) \subset U$ with the properties described in \cite[Theorem~2.3]{ST21}.
\end{proposition}

The compactness theorem for Brakke flows (see \cite{Ilm94} and \cite[Section~3.3]{Ton19}) implies the existence of a limit Brakke flow as $\varepsilon \to 0$.
The proof follows exactly as in \cite[Proposition~4.3]{ST25}.

\begin{proposition}[{\cite[Proposition~4.3]{ST25}}]
    \label{prop:limit_flow}
    For any sequence $\{\varepsilon_j\}_{j=1}^\infty \subset (0,\varepsilon_0]$ converging to $0$, there exists a subsequence (denoted by the same index) and a Brakke flow $\{V_t\}_{t \geq 0}$ with fixed boundary $\partial\Gamma_0$ such that $\lim_{j \to \infty} \wt{V^{\varepsilon_j}_t} = \wt{V_t}$ in $U$ for each $t \geq 0$, $\lim_{t \searrow 0} \wt{V_t} = \wt{V_0} = \haus^n\lfloor_{\Gamma_0}$, and
    \begin{equation}
        \label{eq:energy_ineq}
        \wt{V_t}(U) + \int_0^t\int_U \abs{h(V_t,\cdot)}^2 \,d\wt{V_t}dt \leq \wt{V_0}(U)
    \end{equation}
    for all $t \geq 0$.
\end{proposition}

\section{Brakke's expanding holes lemma}
\label{sec:expanding_holes}

In this section, we prove a slight modification of Brakke's expanding holes lemma \cite[Lemma~6.5]{Bra78}, which is a key ingredient in \cite{ST25}.
Fix $\zeta \in (1/2,1)$ and let $\chi \in C^2_c([0,1))$ be a cut-off function such that
\begin{equation}
    \label{eq:expanding_holes_cutoff}
    0 \leq \chi \leq 1, \qquad \chi \equiv 1 \text{ on } [0,\zeta), \qquad \abs{\chi'} + \abs{\chi''} \leq c(\zeta).
\end{equation}
For $\lambda \geq 0$, define $\Phi_\lambda$ by
\begin{equation}
    \label{eq:expanding_holes_phi}
    \Phi_\lambda(x,t) = \frac{1}{(t+\lambda^2)^{n/2}}\chi^2\biggl(\frac{\abs{x}}{\sqrt{t+\lambda^2}}\biggr).
\end{equation}
Note that, in \cite[Lemma~5.5]{ST25}, the function obtained from $\Phi_\lambda$ by replacing $x$ with $Tx$ is used instead.

The following lemma implies that the area of the hole created in Lemma~\ref{lem:hole_nucleation} grows like $t^{n/2}$ as long as the $L^2$ distance from $T$ is small.

\begin{lemma}
    \label{lem:expanding_holes}
    Let $\lambda \geq 0$, $\tau > 0$, and let $\{V_t\}_{t \in [0,\tau]}$ be a Brakke flow in $U_{\sqrt{\tau + \lambda^2}}$.
    Then, for any $0 < t_1 < t_2 \leq \tau$, we have
    \begin{equation}
        \label{eq:expanding_holes}
        \int \Phi_\lambda(x,t_2) \,d\wt{V_{t_2}}(x) \leq \int \Phi_\lambda(x,t_1) \,d\wt{V_{t_1}}(x) + c\int_{t_1}^{t_2}\int_{U_{\sqrt{t+\lambda^2}}} \frac{\abs{T^\perp x}^2}{(t+\lambda^2)^{n/2+2}} \,d\wt{V_t}(x)dt,
    \end{equation}
    where $c = c(n,\zeta) \in (0,\infty)$.
\end{lemma}

\begin{proof}
    For simplicity, we omit the subscript $\lambda$ in $\Phi_\lambda$.
    A direct computation shows that
    \begin{align*}
        \nabla\Phi(x,t) &= \frac{2\chi\chi'}{(t+\lambda^2)^{(n+1)/2}}\frac{x}{\abs{x}}, \\
        \partial_t\Phi(x,t) &= -\frac{n\chi^2}{2(t+\lambda^2)^{(n+2)/2}} - \frac{\chi\chi'\abs{x}}{(t+\lambda^2)^{(n+3)/2}} = -\frac{n}{2(t+\lambda^2)}\Phi(x,t) - \frac{x}{2(t+\lambda^2)} \cdot \nabla\Phi(x,t).
    \end{align*}
    Hence testing Brakke's inequality \eqref{eq:Brakke_ineq} with $\Phi$ yields
    \begin{equation}
        \label{eq:expanding_holes_eq1}
        \begin{split}
            &\int \Phi(\cdot,t)\,d\wt{V_t}\bigg\rvert_{t=t_1}^{t_2} \leq \int_{t_1}^{t_2}\int (-\Phi h + \nabla\Phi) \cdot h + \partial_t\Phi \,d\wt{V_t}dt \\
            &= \int_{t_1}^{t_2}\int -\Phi\abs{h}^2 + \nabla\Phi \cdot h - \frac{n}{2(t+\lambda^2)}\Phi - \frac{x}{2(t+\lambda^2)} \cdot \nabla\Phi \,d\wt{V_t}(x)dt \\
            &= \int_{t_1}^{t_2}\int -\Phi\abs{h}^2 + \nabla\Phi \cdot h - \frac{n}{2(t+\lambda^2)}\Phi - \frac{Tx}{2(t+\lambda^2)} \cdot \nabla\Phi - \frac{T^\perp x}{2(t+\lambda^2)} \cdot \nabla\Phi \,d\wt{V_t}(x)dt.
        \end{split}
    \end{equation}
    By the definition of the generalized mean curvature, the fourth term becomes
    \begin{align*}
        &\int \frac{Tx}{2(t+\lambda^2)} \cdot \nabla\Phi \,d\wt{V_t}(x) = \int \frac{STx}{2(t+\lambda^2)} \cdot \nabla\Phi \,dV_t(x,S) + \int \frac{S^\perp Tx}{2(t+\lambda^2)} \cdot \nabla\Phi \,dV_t(x,S) \\
        &= \int S \cdot \nabla\biggl(\Phi\frac{Tx}{2(t+\lambda^2)}\biggr) \,dV_t(x,S) - \int \frac{S \cdot T}{2(t+\lambda^2)}\Phi \,d\wt{V_t} + \int \frac{S^\perp Tx}{2(t+\lambda^2)} \cdot \nabla\Phi \,dV_t(x,S) \\
        &= -\int \Phi\frac{Tx}{2(t + \lambda^2)} \cdot h \,d\wt{V_t} - \int \frac{S \cdot T}{2(t+\lambda^2)}\Phi \,d\wt{V_t} + \int \frac{S^\perp Tx}{2(t+\lambda^2)} \cdot \nabla\Phi \,dV_t(x,S).
    \end{align*}
    Plugging this into \eqref{eq:expanding_holes_eq1}, we obtain
    \begin{align*}
        \int \Phi(\cdot,t)\,d\wt{V_t}\bigg\rvert_{t=t_1}^{t_2} &\leq \int_{t_1}^{t_2}\int -\Phi\abs{h}^2 + \nabla\Phi \cdot h + \Phi\frac{Tx}{2(t + \lambda^2)} \cdot h \\
        &\quad - \frac{n - S \cdot T}{2(t + \lambda^2)}\Phi - \frac{S^\perp Tx}{2(t + \lambda^2)} \cdot \nabla\Phi - \frac{T^\perp x}{2(t+\lambda^2)} \cdot \nabla\Phi \,dV_t(x,S)dt \\
        &\leq -\frac{1}{2}\int_{t_1}^{t_2}\int \Phi\abs{h}^2 \,d\wt{V_t}dt \\
        &\quad + c\int_{t_1}^{t_2}\int \frac{\abs{S^\perp\nabla\Phi}^2}{\Phi} - \Phi\frac{S \cdot T^\perp}{t + \lambda^2} + \Phi\frac{\abs{S^\perp Tx}^2}{(t+\lambda^2)^2} + \frac{T^\perp x}{t+\lambda^2} \cdot \nabla\Phi \,dV_t(x,S)dt,
    \end{align*}
    where we used \eqref{eq:perpendicularity}.
    Since
    \begin{gather*}
        \frac{\abs{S^\perp\nabla\Phi}^2}{\Phi} \leq c(n)\abs{\chi'}^2\biggl(\frac{\abs{S-T}^2}{(t+\lambda^2)^{n/2+1}} + \frac{\abs{T^\perp x}^2}{(t+\lambda^2)^{n/2+2}}\biggr), \qquad \Phi\frac{S \cdot T^\perp}{t+\lambda^2} \leq \Phi\frac{\abs{S-T}^2}{2(t+\lambda^2)}, \\
        \Phi\frac{\abs{S^\perp Tx}^2}{(t+\lambda^2)^2} \leq \Phi\frac{\abs{S-T}^2}{t+\lambda^2}, \qquad \biggl\lvert\frac{T^\perp x}{t+\lambda^2} \cdot \nabla\Phi\biggr\rvert \leq 4\chi\abs{\chi'}\frac{\abs{T^\perp x}^2}{(t+\lambda^2)^{n/2+2}},
    \end{gather*}
    the desired inequality follows from \eqref{eq:expanding_holes_cutoff} and \cite[Lemma~5.3 and Remark~5.4]{ST25}.
\end{proof}

\section{Monotonicity formula}
\label{sec:monotonicity}

In this section, we apply Huisken's monotonicity formula to obtain uniform-in-time estimates for the density ratio and the scale invariant $L^2$ distance.
These estimates are crucial for controlling the remainder term in \eqref{eq:expanding_holes}.
For each $y \in \R^{n+1}$ and $s \in \R$, let $\varrho_{(y,s)}$ denote the $n$-dimensional backward heat kernel, defined by
\begin{equation*}
    \varrho_{(y,s)}(x,t) \coloneqq \frac{1}{(4\pi(s-t))^{n/2}}\exp\biggl(-\frac{\abs{x-y}^2}{4(s-t)}\biggr), \qquad (x,t) \in \R^{n+1} \times (-\infty,s).
\end{equation*}
The following proposition is a minor modification of \cite[Propositions~6.2 and 6.4]{KT14}.

\begin{proposition}
    \label{prop:monotonicity}
    Let $R > 0$ and $L \in [2,\infty)$.
    Suppose that $\{V_t\}_{0 \leq t \leq LR^2}$ is a Brakke flow in $U_{LR}$, and $f \in C^2(U_{LR};[0,\infty))$ is a convex function.
    Then, for all $s,t \in [0,LR^2]$ with $t < s$, we have
    \begin{multline*}
        \int_{U_{(L-1)R}} f(x)\varrho_{(0,s)}(x,t) \,d\wt{V_t}(x) \leq \int_{U_{LR}} f(x)\varrho_{(0,s)}(x,0) \,d\wt{V_0}(x) \\
        + c(n)Le^{-L/32}\biggl(\sup_{U_{LR}} (R\abs{\nabla f} + f)\biggr)\sup_{\tau \in [0,t]}\frac{\wt{V_\tau}(U_{LR})}{(LR)^n}.
    \end{multline*}
\end{proposition}

\begin{proof}
    For simplicity, we omit the subscript $(0,s)$ in $\varrho_{(0,s)}$.
    Let $\eta \in C^2_c(U_{LR})$ be a cut-off function such that
    \begin{equation}
        \label{eq:monotonicity_cutoff}
        0 \leq \eta \leq 1, \qquad \eta \equiv 1 \text{ on $U_{(L-1)R}$}, \qquad R\abs{\nabla\eta} + R^2\abs{\nabla^2\eta} \leq c(n).
    \end{equation}
    As in the proof of \cite[Proposition~6.4]{KT14}, testing Brakke's inequality \eqref{eq:Brakke_ineq} with $\phi(x,t) = f(x)\eta(x)\varrho(x,t)$ yields
    \begin{multline*}
        \int f\eta\varrho(\cdot,t) \,d\wt{V_t} - \int f\eta\varrho(\cdot,0) \,d\wt{V_0} \\
        \leq \int_0^t\int -f\eta\varrho\biggl\lvert h - \frac{S^\perp\nabla\varrho}{\varrho}\biggr\rvert^2 - \eta\varrho(S \cdot \nabla^2f) - 2\varrho S(\nabla f) \cdot \nabla\eta - f\varrho(S \cdot \nabla^2\eta) \,dV_\tau(\cdot,S)d\tau,
    \end{multline*}
    where we used the identity $S \cdot \nabla^2\varrho + \abs{S^\perp\nabla\varrho}^2/\varrho + \partial_t\varrho = 0$.
    Since the second term on the right-hand side is non-positive by the convexity of $f$, we obtain
    \begin{align*}
        \int f\eta\varrho(\cdot,t) \,d\wt{V_t} - \int f\eta\varrho(\cdot,0) \,d\wt{V_0} &\leq \int_0^t\int (2\abs{\nabla f}\abs{\nabla\eta} + f\abs{\nabla^2\eta})\varrho \,d\wt{V_\tau}d\tau \\
        &\leq c(n)\int_0^t\int_{U_{LR} \setminus U_{(L-1)R}} (R^{-1}\abs{\nabla f} + R^{-2}f)\varrho \,d\wt{V_\tau}d\tau.
    \end{align*}
    By \eqref{eq:monotonicity_cutoff}, $L \geq 2$, and $s \in [0,LR^2]$,
    \begin{align*}
        \varrho(x,\tau) &\leq \frac{1}{(4\pi(s-\tau))^{n/2}}\exp\biggl(-\frac{(L-1)^2R^2}{4(s-\tau)}\biggr) \leq c(n)(L-1)^{-n}R^{-n}\exp\biggl(-\frac{(L-1)^2R^2}{8(s-\tau)}\biggr) \\
        &\leq c(n)(L-1)^{-n}R^{-n}\exp\biggl(-\frac{(L-1)^2}{8L}\biggr) \leq c(n)e^{-L/32}(LR)^{-n}
    \end{align*}
    for all $x \in U_{LR} \setminus U_{(L-1)R}$ and $\tau \in [0,t]$, which implies the desired inequality.
\end{proof}

The following proposition is a direct consequence of Proposition~\ref{prop:monotonicity}.

\begin{proposition}
    \label{prop:time_uniform}
    Let $L \in [2,\infty)$, and let $\{V^\varepsilon_t\}_{t \geq 0}$ be the Brakke flow obtained in Proposition~\ref{prop:existence}.
    Then there exists $c = c(n) \in (0,\infty)$ such that
    \begin{gather}
        \label{eq:time_uniform_density}
        \frac{\wt{V^\varepsilon_t}(U_R)}{R^n} \leq c\Theta_0 + c\varepsilon_0^{-n}\haus^n(\Gamma_0)Le^{-L/32}, \\
        \label{eq:time_uniform_excess}
        R^{-n-2}\int_{U_R} \abs{T^\perp x}^2 \,d\wt{V^\varepsilon_t}(x) \leq c\mu^2\biggl(1 + \frac{t}{R^2}\biggr) + cR^{-2}\varepsilon_0^{-n+2}\haus^n(\Gamma_0)Le^{-L/32}
    \end{gather}
    for all $\varepsilon \in (0,\varepsilon_0]$, $t \in [0,\varepsilon_0^2/(2L)]$, and $R \in (0,\varepsilon_0/(2L)^{1/2}]$.
\end{proposition}

The fact that the constant $c$ is independent of $\Gamma_0$ allows us to choose $\mu_0$ in Theorem~\ref{thm:main} to depend only on $n$ and $\Theta_0$.

\begin{proof}
    Let $f \in C^2(\R^{n+1})$ be either the constant function $1$ or the function $x \mapsto \abs{T^\perp x}^2$.
    For $t \in [0,\varepsilon_0^2/(2L)]$ and $R \in (0,\varepsilon_0/(2L)^{1/2}]$, we apply Proposition~\ref{prop:monotonicity} with $R = \varepsilon_0/L$ and $s = t+R^2$ to obtain
    \begin{align*}
        &\int_{U_{(1-1/L)\varepsilon_0}} f(x)\varrho_{(0,t+R^2)}(x,t) \,d\wt{V^\varepsilon_t}(x) \\
        &\leq \int_{U_{\varepsilon_0}} f(x)\varrho_{(0,t+R^2)}(x,0) \,d\wt{V^\varepsilon_0}(x) + c(n)Le^{-L/32}\biggl(\sup_{U_{\varepsilon_0}}(\varepsilon_0\abs{\nabla f} + f)\biggr)\sup_{\tau \in [0,t]}\frac{\wt{V^\varepsilon_\tau}(U)}{\varepsilon_0^n} \\
        &\leq \int_{U_{\varepsilon_0}} f(x)\varrho_{(0,t+R^2)}(x,0) \,d\wt{V^\varepsilon_0}(x) + c(n)Le^{-L/32}\varepsilon_0^{-n}\haus^n(\Gamma_0)\sup_{U_{\varepsilon_0}}(\varepsilon_0\abs{\nabla f} + f),
    \end{align*}
    where the last inequality follows from \eqref{eq:energy_ineq}.
    We can estimate the left-hand side from below as
    \begin{equation*}
        \int_{U_{(1-1/L)\varepsilon_0}} f(x)\varrho_{(0,t+R^2)}(x,t) \,d\wt{V^\varepsilon_t}(x) \geq \frac{e^{-1/4}}{(4\pi)^{n/2}}R^{-n}\int_{U_R} f(x) \,d\wt{V^\varepsilon_t}(t).
    \end{equation*}
    For the right-hand side, we first note that Fubini's theorem gives
    \begin{align*}
        &\int_{U_{\varepsilon_0}} f(x)\varrho_{(0,t+R^2)}(x,0) \,d\wt{V^\varepsilon_0}(x) \\
        &= \frac{1}{(4\pi(t+R^2))^{n/2}}\int_{U_{\varepsilon_0}} f(x)\int_{\abs{x}/\sqrt{t+R^2}}^\infty \frac{re^{-r^2/4}}{2} \,drd\wt{V^\varepsilon_0}(x) \\
        &= \frac{1}{(4\pi(t+R^2))^{n/2}}\int_0^\infty \frac{re^{-r^2/4}}{2}\int_{U_{r\sqrt{t+R^2}} \cap U_{\varepsilon_0}} f(x) \,d\wt{V^\varepsilon_0}(x)dr.
    \end{align*}
    In the case $f \equiv 1$, Lemma~\ref{lem:hole_nucleation}~\ref{itm:hole_nucleation_density} yields
    \begin{align*}
        &\int_{U_{\varepsilon_0}} \varrho_{(0,t+R^2)}(x,0) \,d\wt{V^\varepsilon_0}(x) = \frac{1}{(4\pi(t+R^2))^{n/2}}\int_0^\infty \frac{re^{-r^2/4}}{2}\wt{V^\varepsilon_0}(U_{r\sqrt{t+R^2}} \cap U_{\varepsilon_0}) \,dr \\
        &\leq c(n)\int_0^{\varepsilon_0/\sqrt{t+R^2}} \frac{r^{n+1}e^{-r^2/4}}{2}\frac{\wt{V^\varepsilon_0}(U_{r\sqrt{t+R^2}})}{(r\sqrt{t+R^2})^n} \,dr + c(n)\int_{\varepsilon_0/\sqrt{t+R^2}}^\infty \frac{r^{n+1}e^{-r^2/4}}{2}\frac{\wt{V^\varepsilon_0}(U_{\varepsilon_0})}{\varepsilon_0^n} \,dr \\
        &\leq c(n)\Theta_0,
    \end{align*}
    from which \eqref{eq:time_uniform_density} follows.
    In the case $f(x) = \abs{T^\perp x}^2$, Lemma~\ref{lem:hole_nucleation}~\ref{itm:hole_nucleation_excess} yields
    \begin{align*}
        &\int_{U_{\varepsilon_0}} \abs{T^\perp x}^2\varrho_{(0,t+R^2)}(x,0) \,d\wt{V^\varepsilon_0}(x) = \frac{1}{(4\pi(t+R^2))^{n/2}}\int_0^\infty \frac{re^{-r^2/4}}{2}\int_{U_{r\sqrt{t+R^2}} \cap U_{\varepsilon_0}} \abs{T^\perp x}^2 \,d\wt{V^\varepsilon_0}dr \\
        &\leq c(n)(t+R^2)\int_0^{\varepsilon_0/\sqrt{t+R^2}} \frac{r^{n+3}e^{-r^2/4}}{2}\bigl(r\sqrt{t+R^2}\bigr)^{-n-2}\int_{U_{r\sqrt{t+R^2}}} \abs{T^\perp x}^2 \,d\wt{V^\varepsilon_0}(x)dr \\
        &\quad + c(n)(t+R^2)\int_{\varepsilon_0/\sqrt{t+R^2}}^\infty \frac{r^{n+3}e^{-r^2/4}}{2}\varepsilon_0^{-n-2}\int_{U_{\varepsilon_0}} \abs{T^\perp x}^2 \,d\wt{V^\varepsilon_0}(x)dr \\
        &\leq c(n)\mu^2(t+R^2),
    \end{align*}
    from which \eqref{eq:time_uniform_excess} follows.
\end{proof}

\section{Proof of Theorem~\ref{thm:main}}
\label{sec:proof}

In this section, we prove the main theorem.
We fix $\Theta_0$, $\mu$, $U$, $\Gamma_0$, $\{E_{0,i}\}_{i=1}^N$, and $T$ so that Assumption~\ref{asm:tangent_cone} holds.
Let $\mu_1$ and $\varepsilon_0$ be the constants given in Lemma~\ref{lem:hole_nucleation}, and assume that $\mu \in (0,\mu_1)$.
For each $\varepsilon \in (0,\varepsilon_0]$, let $\{V^\varepsilon_t\}_{t \geq 0}$ be the Brakke flow with fixed boundary $\partial\Gamma_0$ and initial datum $\Gamma^\varepsilon_0$ as in Proposition~\ref{prop:existence}.
Fix $\zeta \in (1/2,1)$ and $\alpha \in (1,2)$, and define
\begin{equation*}
    T_\varepsilon = \sup\{t \in [0,\infty) : \wt{V^\varepsilon_\tau}(\Phi_\varepsilon(\cdot,\tau)) \leq \alpha\c \text{ for all $\tau \in [0,t]$}\} \in [0,\infty],
\end{equation*}
where $\Phi_\varepsilon$ is defined as in Section~\ref{sec:expanding_holes} and $\c = \int_T \Phi_1(x,0) \,d\haus^n(x)$.
We begin with the following preliminary lower bound, which will be used to prove that $\limsup_{\varepsilon \to 0} T_\varepsilon > 0$.

\begin{lemma}
    \label{lem:rough_bound}
    There exist $\delta_0 = \delta_0(n,\Theta_0,\zeta,\alpha) \in (0,1)$ and $\varepsilon_1 = \varepsilon_1(\mu,\Gamma_0,T) \in (0,1)$ such that $T_\varepsilon \geq \delta_0\varepsilon^2$ and $\wt{V^\varepsilon_{T_\varepsilon}}(\Phi_\varepsilon(\cdot,T_\varepsilon)) = \alpha\c$ for all $\varepsilon \in (0,\varepsilon_1]$.
\end{lemma}

\begin{proof}
    By Lemma~\ref{lem:expanding_holes}, for all $0 \leq t_1 < t_2$, we have
    \begin{equation}
        \label{eq:rough_bound_1}
        \begin{split}
            \wt{V^\varepsilon_{t_2}}(\Phi_\varepsilon(\cdot,t_2)) - \wt{V^\varepsilon_{t_1}}(\Phi_\varepsilon(\cdot,t_1)) &\leq c(n,\zeta)\int_{t_1}^{t_2}\int_{U_{\sqrt{t+\varepsilon^2}}} \frac{\abs{T^\perp x}^2}{(t + \varepsilon^2)^{n/2+2}} \,d\wt{V^\varepsilon_t}dt \\
            &\leq c(n,\zeta)\int_{t_1}^{t_2} \frac{\wt{V^\varepsilon_t}(U_{\sqrt{t+\varepsilon^2}})}{(t+\varepsilon^2)^{n/2+1}} \,dt.
        \end{split}
    \end{equation}
    It follows from \eqref{eq:rough_bound_1} with $t_1 \nearrow T_\varepsilon$ and $t_2 = T_\varepsilon$ that $\wt{V^\varepsilon_{T_\varepsilon}}(\Phi_\varepsilon(\cdot,T_\varepsilon)) \leq \alpha\c$.
    If $\wt{V^\varepsilon_{T_\varepsilon}}(\Phi_\varepsilon(\cdot,T_\varepsilon)) < \alpha\c$, \eqref{eq:rough_bound_1} with $t_1 = T_\varepsilon$ and $t_2 = T_\varepsilon + \delta$ implies
    \begin{equation*}
        \wt{V^\varepsilon_{t_2}}(\Phi_\varepsilon(\cdot,t_2)) \leq \wt{V^\varepsilon_{T_\varepsilon}}(\Phi_\varepsilon(\cdot,T_\varepsilon)) + c(n,\zeta)\int_{T_\varepsilon}^{T_\varepsilon + \delta} \frac{\wt{V^\varepsilon_t}(U_{\sqrt{t+\varepsilon^2}})}{(t+\varepsilon^2)^{n/2+1}} \,dt < \alpha\c
    \end{equation*}
    for sufficiently small $\delta > 0$.
    This contradicts the definition of $T_\varepsilon$.
    We thus conclude that $\wt{V^\varepsilon_{T_\varepsilon}}(\Phi_\varepsilon(\cdot,T_\varepsilon)) = \alpha\c$.

    By Proposition~\ref{prop:time_uniform}, there exists $L_0 = L_0(n,\varepsilon_0,\Gamma_0) \geq 2$ such that for all $t \in [0,\varepsilon^2_0/(2L_0)]$ with $\sqrt{t+\varepsilon^2} \leq \varepsilon_0/(2L_0)^{1/2}$, we have
    \begin{equation}
        \label{eq:rough_bound_2}
        \frac{\wt{V^\varepsilon_t}(U_{\sqrt{t+\varepsilon^2}})}{(t+\varepsilon^2)^{n/2}} \leq c(n)\Theta_0 + 1.
    \end{equation}
    Set $\varepsilon_1 = \varepsilon_0/(2L_0^{1/2})$.
    Then it follows from \eqref{eq:rough_bound_1} and \eqref{eq:rough_bound_2} that
    \begin{equation}
        \label{eq:rough_bound_3}
        \wt{V^\varepsilon_{t_2}}(\Phi_\varepsilon(\cdot,t_2)) - \wt{V^\varepsilon_{t_1}}(\Phi_\varepsilon(\cdot,t_1)) \leq c(n,\Theta_0,\zeta)\log\biggl(\frac{t_2+\varepsilon^2}{t_1+\varepsilon^2}\biggr)
    \end{equation}
    for all $\varepsilon \in (0,\varepsilon_1]$ and $0 \leq t_1 < t_2 \leq \varepsilon_1^2$.
    Let $\delta_0 \in (0,1)$ be a constant.
    By Lemma~\ref{lem:hole_nucleation}~\ref{itm:hole_nucleation_5} and $\spt\Phi_\varepsilon(\cdot,0) \subset U_\varepsilon$, we have $\wt{V^\varepsilon_0}(\Phi_\varepsilon(\cdot,0)) \leq \c$.
    Hence, for any $t \in [0,\delta_0\varepsilon^2]$, \eqref{eq:rough_bound_3} with $t_1 = 0$ and $t_2 = t$ yields $\wt{V^\varepsilon_t}(\Phi_\varepsilon(\cdot,t)) \leq \c + c\log(1+\delta_0)$.
    Taking $\delta_0$ sufficiently small so that $c\log(1+\delta_0) \leq (\alpha-1)\c$, we obtain $T_\varepsilon \geq \delta_0\varepsilon^2$ for all $\varepsilon \in (0,\varepsilon_1)$.
\end{proof}

The following lemma plays a crucial role in proving that $\limsup_{\varepsilon \to 0} T_\varepsilon > 0$.

\begin{lemma}
    \label{lem:gap}
    For any $E, \Gamma \in (0,\infty)$, $\Lambda \in [1/2,\infty)$, and $\nu \in (0,1)$, there exists a constant $\mu_2 = \mu_2(n,\zeta,\alpha,E,\Gamma,\Lambda,\nu) \in (0,1)$ with the following property.
    Let $\lambda \in [1/2,\Lambda]$, and suppose that $V \in \IV_n(U_\Lambda)$ satisfies
    \begin{gather}
        \label{eq:gap_mass}
        \wt{V}(U_\Lambda) \leq E, \\
        \label{eq:gap_curvature}
        \int_{U_\Lambda} \abs{h(V,x)}^2 \,d\wt{V}(x) \leq \Gamma, \\
        \label{eq:gap_excess}
        \int_{U_\Lambda} \abs{T^\perp x}^2 \,d\wt{V_t}(x) \leq \mu_2^2, \\
        \label{eq:gap_assumption}
        \wt{V}(\Phi_\lambda(\cdot,0)) \leq \alpha\c.
    \end{gather}
    Then we have
    \begin{equation*}
        \wt{V}(\Phi_\lambda(\cdot,0)) \leq (1+\nu)\c.
    \end{equation*}
\end{lemma}

\begin{proof}
    Suppose that the conclusion does not hold.
    Then, for each $j$, there exist $\lambda_j \in [1/2,\Lambda]$ and $V_j \in \IV_n(U_\Lambda)$ satisfying \eqref{eq:gap_mass}--\eqref{eq:gap_assumption} with $\lambda$, $V$, and $\mu_2$ replaced by $\lambda_j$, $V_j$, and $1/j$, respectively.
    Passing to a subsequence, we may assume that $\lambda_j$ converges to some $\lambda \in [1/2,\Lambda]$.
    Moreover, by Allard's compactness theorem \cite[6.4]{All72}, after passing to a further subsequence, there exists $V \in \IV_n(U_\Lambda)$ such that $V_j \to V$ as varifolds.
    By \eqref{eq:gap_excess}, we have $\spt\wt{V} \subset T$.
    Hence there exists a non-negative integer-valued function $\theta$ on $T$ satisfying $V = \var(T,\theta)$.
    For any vector field $\phi \in C^1_c(T \cap U_\Lambda;T)$ on $T$, testing \eqref{eq:mean_curvature} with $g(x) = \phi(x)$ implies that
    \begin{equation*}
        \int_{T \cap U_\Lambda} \theta(x)\operatorname{div}\phi(x) \,d\haus^n(x) = -\int_{T \cap U_\Lambda} \phi(x) \cdot h(V,x) \,d\wt{V}(x) = 0,
    \end{equation*}
    where we used \eqref{eq:perpendicularity}.
    This shows that $\theta$ is constant.
    Since \eqref{eq:gap_assumption} gives $\wt{V}(\Phi_\lambda(\cdot,0)) \leq \alpha\c$, we have $\theta \in \{0, 1\}$.
    Therefore, we obtain $\lim_{j \to \infty} \wt{V_j}(\Phi_\lambda(\cdot,0)) \leq \wt{V}(\Phi_\lambda(\cdot,0)) = \theta\c \leq \c$.
    This contradicts the choice of $V_j$.
\end{proof}

Using Lemmas~\ref{lem:rough_bound} and \ref{lem:gap}, we obtain the following lemma via a blow-up type argument.

\begin{lemma}
    \label{lem:sharp_bound}
    There exists a constant $\mu_3 = \mu_3(n,\Theta_0,\zeta,\alpha) \in (0,1)$ such that if $\mu \in (0,\mu_3)$, then $\limsup_{\varepsilon \to 0} T_\varepsilon > 0$.
\end{lemma}

\begin{proof}
    Suppose, for contradiction, that $\lim_{\varepsilon \to 0} T_\varepsilon = 0$.
    Define the Brakke flow $\{\tilde{V}^\varepsilon_t\}_{t \geq 0}$ in $U_{T_\varepsilon^{-1/2}\varepsilon_0/4}$ by $\tilde{V}^\varepsilon_t = (\eta_{0,T_\varepsilon^{-1/2}})_\sharp V^\varepsilon_{T_\varepsilon t}$.
    Then, for any $L \in [2,\infty)$, Proposition~\ref{prop:time_uniform} gives
    \begin{gather}
        \label{eq:sharp_bound_density}
        \frac{\wt{\tilde{V}^\varepsilon_t}(U_R)}{R^n}= \frac{\wt{V^\varepsilon_{T_\varepsilon t}}(U_{T_\varepsilon^{1/2}R})}{(T_\varepsilon^{1/2}R)^n} \leq c(n)\Theta_0 + c(n)\varepsilon_0^{-n}\haus^n(\Gamma_0)Le^{-L/32}, \\
        \label{eq:sharp_bound_excess}
        \begin{split}
            R^{-n-2}\int_{U_R} \abs{T^\perp x}^2 \,d\wt{\tilde{V}^\varepsilon_t}(x) &= (T_\varepsilon^{1/2}R)^{-n-2}\int_{U_{T_\varepsilon^{1/2}R}} \abs{T^\perp x}^2 \,d\wt{V^\varepsilon_{T_\varepsilon t}}(x) \\
            &\leq c(n)\mu^2\biggl(1 + \frac{t}{R^2}\biggr) + c(n)R^{-2}\varepsilon_0^{-n+2}\haus^n(\Gamma_0)Le^{-L/32}
        \end{split}
    \end{gather}
    for all $\varepsilon \in (0,\varepsilon_0]$, $t \in [0,T_\varepsilon^{-1}\varepsilon_0^2/(2L)]$, and $R \in (0,T_\varepsilon^{-1/2}\varepsilon_0/(2L)^{1/2}]$.
    Hence, by the compactness theorem for Brakke flows, there exist a sequence $\{\varepsilon_j\}_{j=1}^\infty \subset (0,\varepsilon_0]$ converging to $0$ and a Brakke flow $\{\tilde{V}_t\}_{t \geq 0}$ in $\R^{n+1}$ such that $\lim_{j \to \infty} \wt{\tilde{V}^{\varepsilon_j}_t} = \wt{\tilde{V}_t}$ as Radon measures on $\R^{n+1}$ for all $t \geq 0$.
    Furthermore, passing to a further subsequence if necessary, we may assume that $T_{\varepsilon_j}^{-1/2}\varepsilon_j$ converges to some $\lambda \in [0,\delta_0^{-1/2}]$, where $\delta_0$ is the constant given in Lemma~\ref{lem:rough_bound}.
    Letting first $j \to \infty$ and then $L \to \infty$ in \eqref{eq:sharp_bound_density} and \eqref{eq:sharp_bound_excess} with $\varepsilon = \varepsilon_j$, we obtain
    \begin{gather}
        \label{eq:sharp_bound_limit_density}
        \frac{\wt{\tilde{V}_t}(U_R)}{R^n} \leq c(n)\Theta_0, \\
        \label{eq:sharp_bound_limit_excess}
        R^{-n-2}\int_{U_R} \abs{T^\perp x}^2 \,d\wt{\tilde{V}_t}(x) \leq c(n)\mu^2\biggl(1 + \frac{t}{R^2}\biggr)
    \end{gather}
    for all $t \geq 0$ and $R \in (0,\infty)$.
    It follows from \eqref{eq:sharp_bound_limit_density} and \cite[(3.1)]{Ton19} that
    \begin{equation*}
        \int_0^1\int_{U_{(1+\delta_0^{-1})^{1/2}}} \abs{h(\tilde{V}_t,x)}^2 \,d\wt{\tilde{V}_t}(x)dt \leq c(n,\Theta_0,\delta_0).
    \end{equation*}
    In particular, there exists $t_0 \in [1/2,1]$ such that $\tilde{V}_{t_0} \in \IV_n(\R^{n+1})$ and
    \begin{equation*}
        \int_{U_{(1+\delta_0^{-1})^{1/2}}} \abs{h(\tilde{V}_{t_0},x)}^2 \,d\wt{\tilde{V}_{t_0}}(x) \leq c(n,\Theta_0,\delta_0).
    \end{equation*}
    By the definition of $T_\varepsilon$, we have
    \begin{equation*}
        \wt{\tilde{V}_{t_0}}(\Phi_\lambda(\cdot,t_0)) = \lim_{j \to \infty}\wt{\tilde{V}^{\varepsilon_j}_{t_0}}(\Phi_{T_{\varepsilon_j}^{-1/2}\varepsilon_j}(\cdot,t_0)) = \lim_{j \to \infty}\wt{V^{\varepsilon_j}_{T_{\varepsilon_j}t_0}}(\Phi_{\varepsilon_j}(\cdot,T_{\varepsilon_j}t_0)) \leq \alpha\c.
    \end{equation*}
    We choose $\mu_3$ sufficiently small so that Lemma~\ref{lem:gap} applies with $E = c(n)\Theta_0(1+\delta_0^{-1})^{n/2}$, $\Gamma = c(n,\Theta_0,\delta_0)$, $\Lambda = (1+\delta_0^{-1})^{1/2}$, $\nu = (\alpha-1)/2$, and $\mu_2 = \mu_3$.
    Applying Lemma~\ref{lem:gap} to $\tilde{V}_{t_0}$ with $\lambda$ replaced by $(t_0+\lambda^2)^{1/2}$, we obtain $\wt{\tilde{V}_{t_0}}(\Phi_\lambda(\cdot,t_0)) \leq (\alpha+1)\c/2$.
    This, together with Lemma~\ref{lem:expanding_holes} and \eqref{eq:sharp_bound_limit_excess}, implies
    \begin{equation*}
        \wt{\tilde{V}_1}(\Phi_\lambda(\cdot,1)) \leq \wt{\tilde{V}_{t_0}}(\Phi_\lambda(\cdot,t_0)) + c\mu^2 \leq \frac{\alpha+1}{2}\c + c\mu_3^2.
    \end{equation*}
    Hence, replacing $\mu_3$ by a smaller number if necessary, we conclude that $\wt{\tilde{V}_1}(\Phi_\lambda(\cdot,1)) < \alpha\c$.
    However, $\wt{\tilde{V}_1}(\Phi_\lambda(\cdot,1)) = \lim_{j \to \infty}\wt{V^{\varepsilon_j}_{T_{\varepsilon_j}}}(\Phi_{\varepsilon_j}(\cdot,T_{\varepsilon_j})) = \alpha\c$ by Lemma~\ref{lem:rough_bound}.
    This is a contradiction.
\end{proof}

We are now ready to prove Theorem~\ref{thm:main}.

\begin{proof}[Proof of Theorem~\ref{thm:main}]
    By Allard's regularity theorem \cite[p.466]{All72} and $\Theta^n(\wt{V_0},0) > 1$, there exists $\eta = \eta(n) \in (0,1)$ such that $\Theta^n(\wt{V_0},0) \geq 1+\eta$.
    We first fix $R_0 = R_0(\Gamma_0) \in (0,\dist(0,\partial U))$ such that
    \begin{equation*}
        \frac{\haus^n(\Gamma_0 \cap U_R)}{R^n} \geq \biggl(1 + \frac{3}{4}\eta\biggr)\omega_n \geq \biggl(1 + \frac{3}{4}\eta\biggr)\c
    \end{equation*}
    for all $R \in (0,R_0)$.
    We then choose $\zeta = \zeta(n) \in (1/2,1)$, which appears in \eqref{eq:expanding_holes_cutoff}, so that
    \begin{equation}
        \label{eq:proof_1}
        \int_{\Gamma_0} \Phi_0(\cdot,R^2) \,d\haus^n \geq \zeta^n\frac{\haus^n(\Gamma_0 \cap U_{\zeta R})}{(\zeta R)^n} \geq \zeta^n\biggl(1 + \frac{3}{4}\eta\biggr)\c > \biggl(1 + \frac{\eta}{2}\biggr)\c = \alpha\c.
    \end{equation}
    Let $\mu_0 \in (0,1)$ be the constant $\mu_3$ given by Lemma~\ref{lem:sharp_bound} with $\alpha = 1+\eta/2$.
    By Lemma~\ref{lem:sharp_bound}, there exists a sequence $\{\varepsilon_j\}_{j=1}^\infty \subset (0,\varepsilon_0]$ and $T_0 \in (0,R_0^2]$ such that $\lim_{j \to \infty}\varepsilon_j = 0$ and $\liminf_{j \to \infty} T_{\varepsilon_j} > T_0$.
    Then, applying Proposition~\ref{prop:limit_flow}, we obtain a subsequence (denoted by the same index) and a Brakke flow $\{V_t\}_{t \geq 0}$ with fixed boundary $\partial\Gamma_0$ such that $\lim_{j \to 0} \wt{V^{\varepsilon_j}_t} = \wt{V_t}$ in $U$ for each $t \geq 0$, $\lim_{t \searrow 0} \wt{V_t} = \wt{V_0} = \haus^n\lfloor_{\Gamma_0}$, and \eqref{eq:energy_ineq} holds.
    It follows from the definition of $T_\varepsilon$ and \eqref{eq:proof_1} that
    \begin{equation}
        \label{eq:proof_2}
        \wt{V_t}(\Phi_0(\cdot,t)) \leq \alpha\c < \wt{V_0}(\Phi_0(\cdot,t))
    \end{equation}
    for all $t \in (0,T_0]$.
    Suppose, for contradiction, that there exists $t_0 > 0$ such that $\wt{V_{t_0}}(U) \geq \wt{V_0}(U)$.
    By \eqref{eq:energy_ineq}, we have $\wt{V_{t_0}}(U) = \wt{V_0}(U)$ and $h(V_t,\cdot) = 0$ for a.e. $t \in [0,t_0]$.
    This together with Brakke's inequality \eqref{eq:Brakke_ineq} implies
    \begin{equation}
        \label{eq:proof_3}
        \wt{V_{t_2}}(\phi) \leq \wt{V_{t_1}}(\phi)
    \end{equation}
    for all $\phi \in C^1_c(U;[0,\infty))$ and $0 \leq t_1 < t_2 \leq t_0$.
    Set $s_0 = \min\{t_0,T_0\}$.
    Combining \eqref{eq:proof_2} and \eqref{eq:proof_3} yields
    \begin{equation*}
        \wt{V_{t_0}}(\Phi_0(\cdot,s_0)) \leq \wt{V_{s_0}}(\Phi_0(\cdot,s_0)) < \wt{V_0}(\Phi_0(\cdot,s_0)).
    \end{equation*}
    Hence \eqref{eq:proof_2} and this give $\wt{V_{t_0}}(1-\Phi_0(\cdot,s_0)) > \wt{V_0}(1-\Phi_0(\cdot,s_0))$.
    Then by approximation, one can show that there exists $\psi \in C^1_c(U;[0,\infty))$ such that $\wt{V_{t_0}}(\psi) > \wt{V_0}(\psi)$, which contradicts \eqref{eq:proof_3}.
\end{proof}

\bibliography{references}
\bibliographystyle{alpha}
\end{document}